\documentclass[11pt]{article}
\usepackage{bbm}
\usepackage[all,cmtip]{xy}
\usepackage{amsfonts,amssymb,amsmath,amscd,amsthm,mathtools}
\usepackage{mathrsfs}
\usepackage{latexsym}
\usepackage{color}
\usepackage{hyperref}
\usepackage{verbatim,enumitem}
\usepackage{tikz-cd}
\usepackage[OT1]{fontenc}
\usepackage{booktabs} 
\input diagxy
\usepackage[total={155mm,230mm}]{geometry}
\DeclareMathOperator*{\colim}{colim}

\DeclareMathOperator{\cpt}{\pitchfork}

\theoremstyle{plain}
  \newtheorem{thm}{Theorem}[section]
  \newtheorem{lem}[thm]{Lemma}
  \newtheorem{prop}[thm]{Proposition}
  \newtheorem{cor}[thm]{Corollary}
\theoremstyle{definition}
  \newtheorem{defn}[thm]{Definition}
  \newtheorem{example}[thm]{Example}
  \newtheorem{rem}[thm]{Remark}

\newtheorem{con}[thm]{Convention}

\newcommand{\ra}{\rightarrow}
\newcommand{\lra}{\longrightarrow}

\newcommand{\dl}{{([0,\infty],d_L)}}
\newcommand{\dr}{{([0,\infty],d_R)}}

\newcommand{\sy}{{\sf y}}

\newcommand{\bv}{\bigvee}

\def\oto{{\bfig\morphism<180,0>[\mkern-4mu`\mkern-4mu;]\place(86,0)[\circ]\efig}}

\begin{document}
\title{A Hofmann-Mislove theorem for approach spaces}

\author{Junche Yu, Dexue Zhang}
\date{}
\maketitle

\begin{abstract}
The Hofmann-Mislove theorem says that the ordered set of open filters of the open-set lattice of a sober topological space is isomorphic to the  ordered set of compact saturated sets (ordered by reverse inclusion) of that space. This paper concerns a metric analogy of this result. To this end, the notion of compact functions of approach spaces  is introduced. Such functions are an analog of compact subsets in the enriched context.  It is shown that for  a sober approach space $X$, the metric space  of proper open $[0,\infty]$-filters of the metric space of   upper regular functions of $X$ is isomorphic to the opposite of the metric space of   inhabited and saturated compact functions of $X$,   establishing a Hofmann-Mislove theorem for approach spaces.

\vskip 3pt


\end{abstract}

\section{Introduction}

While metric spaces (not necessarily symmetric and finitary) are  ordered sets (or, categories) enriched over Lawvere's quantale $([0,\infty],\geq,+,0)$ \cite{Lawvere1973}, approach spaces introduced by Lowen \cite{RL89} are topological spaces enriched over $([0,\infty],\geq,+,0)$. Approach spaces are  among the prime examples in  \emph{Monoidal Topology} (see e.g. \cite{CHT,Hof2007,Hof2011,Monoidal top}). For an exposition of these spaces, we refer to the monograph   \cite{Lowen15} by their inventor.

This paper concerns the connections between metric spaces and approach spaces, such connections are an important ingredient of Quantitative Domain Theory which, roughly speaking, replaces partially ordered sets  in Domain Theory \cite{Scott72,Gierz2003} by metric spaces  and more generally, quasi-uniform spaces, see e.g. \cite{BvBR1998,Goubault,Smyth88,Wagner97}. It has become clear that a coherent theory,  parallel to that of domains, in the $[0,\infty]$-enriched context is needed; see \cite{AW2011,GH,HW2012,LiZ18b,Windels}  for some efforts in this regard.  Of course, in order that parallels emerge, necessary modifications have to be made in the $[0,\infty]$-enriched context;  these modifications often help us understand the subject better.

An important result in domain theory, known as the Hofmann-Mislove theorem \cite{HM1981,Gierz2003}, says that the ordered set of   (proper) open filters of the open-set lattice of a sober topological space $X$ is isomorphic to  the ordered set of  (non-empty) compact saturated sets of  $X$ (ordered by reverse inclusion), hence establishes a connection between objects of an order-theoretic nature to objects of a topological nature. The aim of this paper is to look for an analogous result in the $[0,\infty]$-enriched context.  In order to achieve this, we have to determine, in the $[0,\infty]$-enriched context, an analog of compact subset and an analog of open filter of open-set lattice.

In \cite{AW2011} a notion of open $[0,\infty]$-filter of metric spaces is introduced. There is a natural metric $\rho_X$ on the set $\mathfrak{U}X$ of upper regular functions of an approach space, so, we take open $[0,\infty]$-filters of the metric space $(\mathfrak{U}X, \rho_X)$ for the $[0,\infty]$-enriched version of open filters of open-set lattices. It remains to determine an analog of compact subset in the enriched context. To this end,  the notion of compact functions of approach spaces is introduced (Definition \ref{cf}) and some basic properties of such functions are investigated. The results show that such functions behave in a way parallel to that of compact subsets of topological spaces.
\begin{center} \begin{tabular}{  l l }
\toprule ordered set & metric space\\
\midrule topological space & approach space \\
\midrule
   compact subset   & compact function \\ \midrule  lattice of open sets & metric space of upper regular functions \\ \midrule   open filter of a partially ordered set ~ & open $[0,\infty]$-filter of a metric space
  \\ \midrule  sober topological space & sober approach space   \\
\bottomrule \end{tabular}\end{center}

With the notions of compact functions and open $[0,\infty]$-filters at hand, we establish an enriched version of the Hofmann-Mislove theorem in the last section, which says that for a sober approach space $(X,\delta)$, the metric space of proper open $[0,\infty]$-filters of the metric space $(\mathfrak{U}X,\rho_X)$ of upper regular functions of $(X,\delta)$ is isomorphic to the opposite of the metric space of inhabited  compact saturated functions of $(X,\delta)$.

\section{$[0,\infty]$-ideals  and $[0,\infty]$-filters  of metric spaces}

In this section we recall some basic notions about metric spaces, viewed as categories enriched over Lawvere's quantale $([0,\infty],\geq,+,0)$  \cite{Lawvere1973,Borceux1994}, the aim is to fix terminologies and notations. Only  Proposition \ref{real-valued ideal} and Proposition \ref{real-valued filter} are new.

A metric space is a pair $(X,d)$ consisting of a set $X$ and a map $d\colon X\times X\lra[0,\infty]$, called a metric on $X$, such that $d(x,x)=0$ and $d(x,y)+d(y,z)\geq d(x,z)$ for all $x,y,z\in X$. For each metric $d$ on $X$, $d^{\rm op}(x,y)\coloneqq d(y,x)$ is also a metric on $X$, called the opposite of $d$. A metric $d$ is \begin{itemize}\setlength{\itemsep}{0pt} \item symmetric, if $d=d^{\rm op}$; that is, $d(x,y)=d(y,x)$ for all $x,y\in X$;  \item separated, if $x=y$ whenever $d(x,y)=d(y,x)=0$;  \item finitary, if $d(x,y)<\infty$ for all $x,y\in X$. \end{itemize}
A ``classical metric''   is just a symmetric, separated and finitary one.

A map $f\colon (X,d_X)\lra (Y,d_Y)$ between metric spaces is \emph{non-expansive} if, for all $x,y\in X$,
$$d_X(x,y)\geq d_Y(f(x),f(y)).$$
A map $f\colon (X,d_X)\lra (Y,d_Y)$ between metric spaces is \emph{isometric} if, for all $x,y\in X$,
$$d_X(x,y)= d_Y(f(x),f(y)).$$ Non-expansive maps are just $[0,\infty]$-functors when metric spaces are viewed as $[0,\infty]$-categories.
Metric spaces and non-expansive maps form a category $$\sf Met.$$

For all $a,b\in [0,\infty]$, let
$$d_L(a,b)=b\ominus a\coloneqq\max\{0,b-a\}.$$
Then $d_L$ is a metric on $[0,\infty]$. The opposite of $d_L$ is denoted by $d_R$; that is,
$$d_R(a,b)=a\ominus b.$$
Both $\dl$ an $\dr$ are separated, non-symmetric, and non-finitary metric spaces. We note here that the operations $\ominus, +\colon[0,\infty]\times[0,\infty]\lra[0,\infty]$ are interlocked by the adjoint property:
$$a\ominus b\leq c\iff a\leq b+c$$
for all $a,b,c\in [0,\infty]$.

The category $\sf Met$ is complete and cocomplete.
In particular, for every set $X$, the product metric on the $X$-power of $\dl$ is given by
$$\rho_X\colon [0,\infty]^X\times [0,\infty]^X\lra [0,\infty], \quad \rho_X(\lambda,\mu)=\sup_{x\in X}\mu(x)\ominus \lambda(x).$$

The \emph{underlying order} $\sqsubseteq_d$ of a metric space $(X,d)$ refers to
the order relation on $X$ defined by $$ x\sqsubseteq_d y \quad \text{if $d(x,y)=0.$}$$ It is clear that a metric space is separated if and only if its underlying order is a partial order and that any non-expansive map preserves the underlying order.

The underlying order $\sqsubseteq$ of  $([0,\infty]^X,\rho_X)$ is given by \[\lambda\sqsubseteq \mu\iff \forall x\in X,\lambda(x)\geq \mu(x).\] Particularly, the underlying order of $([0,\infty],d_L)$ is opposite to the usual less-than-or-equal-to relation.

\begin{con}Some conventions and notations. \begin{enumerate}[label=\rm(\alph*)] \setlength{\itemsep}{0pt}
\item When talking about a subset $A$ of $[0,\infty]^X$, we always assume that $A$ is endowed with the metric inherited from the space $([0,\infty]^X,\rho_X)$, which will also be denoted by  $\rho_X$.
\item  Though the underlying order of $([0,\infty]^X,\rho_X)$ is opposite to the \emph{usual} pointwise order, for each subset $\{\lambda_i\}_i$ of $[0,\infty]^X$, we still write   $\inf_i \lambda_i$ and $\sup_i \lambda_i$, respectively, for its greatest lower bound and least upper bound  with respect to the usual pointwise order.
\item For each $r\in[0,\infty]$ and $A\subseteq X$, we write $r_A$ for the map $X\lra[0,\infty]$ given by \[r_A(x)= \begin{cases} r, & x\in A,\\ \infty, & x\notin A. \end{cases} \]
    \end{enumerate}   \end{con}

A distributor $\phi\colon (X,d_X)\oto (Y,d_Y)$ is a map $\phi\colon X\times Y\to[0,\infty]$ such that $$d_Y(y,y')+\phi(x,y)+d_X(x',x)\geq \phi(x',y')$$ for all $x,x'\in X$ and $y,y'\in Y$.

For distributors $\phi\colon (X,d_X)\oto(Y,d_Y)$ and $\psi\colon(Y,d_Y)\oto (Z,d_Z)$, the composite  $\psi\circ\phi$ refers to the distributor $(X,d_X)\oto(Z,d_Z)$ given by $$\psi\circ\phi(x,z)=\inf_{y\in Y}(\psi(y,z)+\phi(x,y)).$$

For each non-expansive map $f\colon (X,d_X)\to(Y,d_Y)$, by the graph of $f$ we mean the distributor  $$f_*\colon(X,d_X)\oto(Y,d_Y), \quad (x,y)\mapsto d_Y(f(x),y);$$  by the cograph of $f$ we mean the distributor
$$f^*\colon(Y,d_Y)\oto(X,d_X), \quad (y,x)\mapsto d_Y(y,f(x)).$$  It is well-known that the graph $f_*$ is left adjoint to the cograph $f^*$ in the sense that $$f^*\circ f_*(x,x')\leq d_X(x,x')\quad \text{and}\quad f_*\circ f^*(y,y')\geq d_Y(y,y')$$ for all $x,x'\in X$ and $y,y'\in Y$.

Let $(X,d)$ be a metric space. A $\emph{weight}$ of $(X,d)$ is a distributor $\phi\colon(X,d)\oto\star$, where $\star$ denotes the singleton metric space; in other words, a weight of $(X,d)$ is a map $\phi\colon X\lra[0,\infty]$ such that for all $x,y\in X$,
$$\phi(y)+d(x,y)\geq\phi(x).$$
A \emph{coweight} of $(X,d)$ is a distributor $\psi\colon\star\oto(X,d)$; in other words, a coweight of $(X,d)$ is a map $\lambda\colon X\lra[0,\infty]$ such that for all $x,y\in X$,
$$\lambda(x)+d(x,y)\geq\lambda(y).$$

It is clear that a weight of $X$ is just a non-expansive map $\phi\colon (X,d^{\rm op})\lra\dl$;  a coweight of $X$ is a non-expansive map $\lambda\colon (X,d)\lra\dl$.

Write $$\mathcal{P}X$$ for the set of all weights of $(X,d)$. For each element $x$ of a metric space $(X,d)$, the map
$$d(-,x)\colon X\lra[0,\infty] $$
is a weight of $(X,d)$.   Weights  of this form are said to be \emph{representable}.

\begin{lem} {\rm(Yoneda lemma)} \label{yoneda lemma}
Let $(X,d)$ be a metric space. Then for each point $x$ and each weight  $\phi$  of $(X,d)$,
$$\rho_X(d(-,x),\phi)=\phi(x).$$
\end{lem}
The map
$${\sy}\colon (X,d)\lra (\mathcal{P}X,\rho_X), \quad \sy(x)=d(-,x)$$
 is isometric by the Yoneda lemma, and is called the \emph{Yoneda embedding}.

Assume that     $f\colon (X,d_X)\to (Y,d_Y)$ and $g\colon (Y,d_Y)\to (X,d_X)$ are non-expansive maps. We say that $f$ is left adjoint to $g$, or $g$ is right adjoint to $f$, and write $f\dashv g$, if for all $x\in X$ and $y\in Y$,
$$d_Y(f(x),y)=d_X(x,g(y)).$$

Each non-expansive map $f\colon (X,d_X)\to (Y,d_Y)$ gives rise to an adjunction between the metric spaces $(\mathcal{P}X,\rho_X)$ and $(\mathcal{P}Y,\rho_Y)$. Precisely, the map
$$f^\rightarrow\colon (\mathcal{P}X,\rho_X)\lra(\mathcal{P}Y,\rho_Y), \quad f^{\rightarrow}(\phi)=\phi\circ f^*$$
is left adjoint to
$$f^\leftarrow\colon (\mathcal{P}Y,\rho_Y)\lra(\mathcal{P}X,\rho_X), \quad f^\leftarrow(\psi)=\psi\circ f_*.$$

Let $(X,d)$ be a metric space. We say that an element $a$ of $X$ is   a \emph{colimit} of a weight $\phi$ and write $a=\colim\phi$ if
$$d(a,y)=\rho_X(\phi,d(-,y))$$
for all $y\in X$. Any colimits $a$ and $b$  of the same weight  are  isomorphic in the sense that $d(a,b)=d(b,a)=0$. Dually, we say that  $a$   is a \emph{limit} of a coweight $\psi$ and write $a=\lim\psi$  if
$$d(y,a)=\rho_X(\psi,d(y,-))$$
for all $y\in X$.

A metric space $(X,d)$ is \emph{cocomplete} if every weight of $(X,d)$ has a colimit. One readily verifies that $(X,d)$ is cocomplete if and only if the Yoneda embedding $\sy\colon (X,d)\lra(\mathcal{P}X,\rho_X)$ has a left adjoint.
It is   known (see e.g. \cite{Stubbe2005}) that $(X,d)$ is cocomplete if and only if $(X,d)$ is \emph{complete} in the sense that every coweight has a limit.

Let $(X,d)$ be a metric space. For each $x\in X$ and  $r\in [0,\infty]$, the \emph{tensor of $r$ with $x$}, denoted by $r\otimes x$, is a point of $X$ such that for all $y\in X$, \[d(r\otimes x,y)= d(x,y)\ominus r.\]
The \emph{cotensor of $r$ with $x$}, is an element $r\multimap x$ of $X$ such that for all $y\in X$,
$$d(y,r\multimap x)=d(y,x)\ominus r.$$

A metric space $(X,d)$ is   \emph{tensored} if the tensor $r\otimes x$ exists for all $x\in X$ and $r\in [0,\infty]$.

\begin{prop}\label{non-expansive via tensor}
If $(X,d_X)$ and $(Y,d_Y)$  are both tensored metric spaces, then a map $f\colon (X,d_X)\lra (Y,d_Y)$ is non-expansive  if  and only if \begin{enumerate}[label=\rm(\roman*)] \setlength{\itemsep}{0pt}
  \item $f\colon (X,\sqsubseteq_{d_X})\lra (Y,\sqsubseteq_{d_Y})$ preserves order; and \item  for all $x\in X$ and $r\in [0,\infty]$, $r\otimes f(x)\sqsubseteq_{d_Y} f(r\otimes x)$.\end{enumerate}\end{prop}

\begin{proof}See e.g. \cite[Proposition 2.10]{GH}.\end{proof}

Every cocomplete metric space is tensored: $r\otimes x$ is the colimit of the weight $r+ d(-,x)$.
The following characterization of cocomplete (hence complete) metric spaces is   a special case of  \cite[Theorem 2.7]{Stubbe2006}, the conditions (ii) and (iii) amount to that $(X,d)$, as an enriched category, is conically cocomplete.
\begin{prop}\label{cocomplete via tensor}  A metric space $(X,d)$ is cocomplete if and only if \begin{enumerate}[label=\rm(\roman*)] \setlength{\itemsep}{0pt}

\item $(X,d)$ is tensored;   \item each subset $A$ of $X$ has a join   in $(X,\sqsubseteq_d)$; and \item   for each $x\in X$ and each   $A\subseteq X$,       $d(\bv A,x)= \sup\limits_{a\in A}d(a,x)$, where $\bv A$ refers to the join of $A$ in $(X,\sqsubseteq_d)$.    \end{enumerate} \end{prop}

\begin{example} (See e.g. \cite{GH}) For each set $X$, the metric space $([0,\infty]^X,\rho_X)$ is cocomplete, hence complete. In particular, for all $\lambda\in[0,\infty]^X$ and $r\in[0,\infty]$, $r\otimes\lambda= r+\lambda$ and $r\multimap \lambda=\lambda\ominus r$.
\end{example}

For each set $X$, define a function $$\cpt\colon [0,\infty]^X\times[0,\infty]^X\lra[0,\infty]$$  by $$\phi\cpt\lambda=\inf_{x\in X}(\phi(x)+\lambda(x))$$ for all $\phi,\lambda\in[0,\infty]^X$. This function   plays an important role in this paper.

\begin{rem}For each set $X$ and  $\phi,\lambda\in [0,\infty]^X$, if we view $X$ as a discrete metric space, view $\phi$ as a weight and $\lambda$ as a coweight of $X$, then $\phi \pitchfork \lambda$ is essentially the composite distributor $\phi\circ\lambda\colon \star\oto\star$.
\end{rem}
\begin{lem}\label{sub vs tensor}
Let $X$ be a set. Then for all $\phi,\lambda\in [0,\infty]^X$ and $b\in[0,\infty]$,
\[\rho_X(\phi, b\ominus\lambda)=b\ominus(\lambda\cpt\phi).\]
\end{lem}

\begin{proof}
Routine calculations.\end{proof}


Let  $\lambda\in[0,\infty]^X$. We say that \begin{enumerate}[label=\rm(\roman*)] \setlength{\itemsep}{0pt}
  \item  $\lambda$ is  inhabited   if $\inf\limits_{x\in X}\lambda(x)=0$;
  \item  $\lambda$ is  vacuous   if $\lambda=\infty_X$.\end{enumerate}
It is clear that $\lambda$ is inhabited if and only if $\lambda\cpt 0_X=0$,  $\lambda$ is vacuous if and only if $\lambda\cpt \mu=\infty$ for all $\mu\in[0,\infty]^X$.

Let $(X,d)$ be a metric space. A net $\{x_i\}_i$ in $(X,d)$ is  \emph{forward Cauchy} \cite{Smyth88,BvBR1998} if
$$\inf_i\sup_{k\geq j\geq i}d(x_j,x_k)=0.$$

\begin{defn}(\cite{FS2002,AW2011}) A $[0,\infty]$-ideal of a metric space $(X,d)$ is a weight $\phi$ such that $$\phi=\inf_i\sup_{j\geq i}d(-,x_j)$$
for some forward Cauchy net $\{x_i\}_i$.\end{defn}

The following theorem characterizes $[0,\infty]$-ideals of a metric space $(X,d)$ via their relation to the weights and coweights of $(X,d)$.

\begin{thm}\label{irideal}
Let $\phi$ be a weight of a metric space $(X,d)$. The following are equivalent:  \begin{enumerate}[label=\rm(\arabic*)] \setlength{\itemsep}{0pt}
  \item $\phi$ is a $[0,\infty]$-ideal. \item $\phi$ is inhabited and is irreducible in the sense that for all weights $\phi_1,\phi_2$ of $(X,d)$,  $$\inf\{\phi_1,\phi_2\}\leq\phi\implies \text{ either $\phi_1\leq\phi$ or $\phi_2\leq\phi$}.$$ \item $\phi$ is inhabited and is flat in the sense that for all coweights $\lambda_1,\lambda_2$ of $(X,d)$,  $$\phi\circ\sup\{\lambda_1,\lambda_2\}=\max\{\phi\circ\lambda_1, \phi\circ\lambda_2\}.$$\end{enumerate}
\end{thm}
\begin{proof}The equivalence $(1)\Leftrightarrow(3)$ is contained in the proof of \cite[Theorem 7.15]{SV2005};  the equivalence $(1)\Leftrightarrow(2)$ is a special case of \cite[Theorem 3.13]{LZZ2020}. \end{proof}

Since each ideal $D$ of $(X,\sqsubseteq_d)$ can be viewed as a forward Cauchy net of $(X,d)$,  $$d(-,D)\coloneqq\inf_{x\in D}d(-,x)$$ is a $[0,\infty]$-ideal of $(X,d)$. For a cocomplete metric space, every $[0,\infty]$-ideal is   of this form.

\begin{prop}\label{colimit as join for ideal} {\rm (\cite[Proposition 4.8]{LZ2020})} Let $(X,d)$ be a cocomplete metric space and let $\phi$ be a $[0,\infty]$-ideal of $(X,d)$. Then \begin{enumerate}[label=\rm(\roman*)] \setlength{\itemsep}{0pt}
  \item $\Delta(\phi)\coloneqq\{x\in X\mid\phi(x)=0\}$ is an ideal in $(X,\sqsubseteq_d)$ and $\phi=\inf\limits_{x\in \Delta(\phi)}d(-,x)$.
      \item  Colimits of $\phi$ in $(X,d)$ are precisely  joins of $\Delta(\phi)$ in $(X,\sqsubseteq_d)$.\end{enumerate}
\end{prop}

The following characterization of $[0,\infty]$-ideals of a cocomplete metric space and its dual, Proposition \ref{real-valued filter}, are useful in this paper.
\begin{prop} \label{real-valued ideal}Let $(X,d)$ be a cocomplete metric space. Then a function $\phi\colon X\lra[0,\infty]$   is a  $[0,\infty]$-ideal if and only if it satisfies: \begin{enumerate}[label=\rm(\roman*)] \setlength{\itemsep}{0pt}
   \item   for all $x\in X$ and $r\in[0,\infty]$,  $\phi(r\otimes x)=\phi(x)\ominus r$;

\item  $\phi(x\vee y)= \max\{\phi(x),\phi(y)\}$ for all $x,y\in X$, where $x\vee y$ is a join of $x$ and $y$ in $(X,\sqsubseteq_d)$. \end{enumerate}
\end{prop}

\begin{proof} If $\phi$ is a $[0,\infty]$-ideal,  there is   a forward Cauchy net $\{x_i\}_{i\in I}$ of $(X,d)$ such that $$\phi=\inf_i\sup_{j\geq i}d(-,x_j).$$ Now we check that it satisfies (i) and (ii).

(i) We calculate:
\begin{align*}
\phi(r\otimes x) & = \inf_i\sup_{j\geq i}d(r\otimes x,x_j)\\
& = \inf_i\sup_{j\geq i}(d(x,x_j)\ominus r)\\
& = (\inf_i\sup_{j\geq i}d(x,x_j))\ominus r\\
& = \phi(x)\ominus r.
\end{align*}

(ii)  We calculate:
\begin{align*}
\phi(x\vee y) & = \inf_i\sup_{j\geq i}d(x\vee y,x_j)\\
& = \inf_i\sup_{j\geq i} \max\{d(x,x_j),d(y,x_j)\} & \text{(Proposition \ref{cocomplete via tensor})}\\
& = \inf_i\max\{\sup_{j\geq i} d(x,x_j),\sup_{j\geq i}d(y,x_j)\}\\
& =\max\{\inf_i\sup_{j\geq i}d(x,x_j),\inf_i\sup_{j\geq i}d(y,x_j)\}&\text{(the index set is directed)}\\
& = \max\{\phi(x),\phi(y)\}.
\end{align*}

As for sufficiency, assume that $\phi\colon X\lra[0,\infty]$  satisfies (i) and (ii). First, by (i), (ii) and Proposition \ref{non-expansive via tensor}
one sees that $\phi\colon(X,d^{\rm op})\lra([0,\infty],d_L)$ is non-expansive, hence a  weight of $(X,d)$. Next, let
$$\Delta(\phi)=\{x\in X\mid\phi(x)=0\}.$$
By (ii) one sees that $\Delta(\phi)$ is an ideal of $(X,\sqsubseteq_d)$. Now we show that for each $x\in X$,
$$\phi(x)=\inf_{a\in \Delta(\phi)}d(x,a),$$
which implies that $\phi$ is a $[0,\infty]$-ideal  since $\Delta(\phi)$ can be viewed as a forward Cauchy net of $(X,d)$.
On one hand,  for each $a\in \Delta(\phi)$, since $\phi(x)\leq \phi(a)+d(x,a)=d(x,a)$, it follows that $$\phi(x)\leq\inf_{a\in \Delta(\phi)}d(x,a).$$ On the other hand, since
$\phi(\phi(x)\otimes x)=\phi(x)\ominus\phi(x)=0,$ then
$\phi(x)\otimes x\in \Delta(\phi)$. Because
$$0=d(\phi(x)\otimes x, \phi(x)\otimes x)= d(x,\phi(x)\otimes x)\ominus \phi(x),$$ therefore
$\phi(x)\geq d(x,\phi(x)\otimes x)$, hence $$\phi(x) \geq\inf_{a\in \Delta(\phi)}d(x,a),$$ which completes the proof.
\end{proof}

Following  \cite{AW2011}, a map $\lambda\colon X\lra[0,\infty]$ is called a \emph{$[0,\infty]$-filter} of $(X,d)$  if it is a $[0,\infty]$-ideal of $(X,d^{\rm op})$. In other words, a $[0,\infty]$-filter of $(X,d)$ is a coweight $\lambda$ such that
$$\lambda=\inf_i\sup_{j\geq i}d(x_j,-)$$
for some   net $\{x_i\}_{i\in I}$  that  is \emph{backward Cauchy} in the sense that
$$\inf_i\sup_{k\geq j\geq i}d(x_k,x_j)=0.$$

The following proposition is dual to Proposition \ref{real-valued ideal}.
\begin{prop} \label{real-valued filter} Let $(X,d)$ be a complete metric space. Then a function $\lambda\colon X\lra[0,\infty]$   is a  $[0,\infty]$-filter if and only if it satisfies: \begin{enumerate}[label=\rm(\roman*)] \setlength{\itemsep}{0pt}
   \item   for all $x\in X$ and $r\in[0,\infty]$,  $\lambda(r\multimap x)=\lambda(x)\ominus r$;

\item  $\lambda(x\land y)= \max\{\lambda(x),\lambda(y)\}$ for all $x,y\in X$, where   $x\land y$ is a meet of $x$ and $y$ in $(X,\sqsubseteq_d)$. \end{enumerate}
\end{prop}

\begin{rem}\label{finite weighted limits} Let $(X,d)$ be a complete (hence cocomplete) metric space. Then, in the terminology of enriched category theory \cite{Borceux1994}, both  $r\multimap x$ and   $x\wedge y$ are finite weighted limits in $(X,d)$, i.e., weighted limit of a functor from a finite $[0,\infty]$-enriched category  to $(X,d)$.   In this perspective, a $[0,\infty]$-filter   is  just  a $[0,\infty]$-functor $\lambda\colon (X,d)\lra ([0,\infty],d_L)$ that preserves finite weighted limits; a $[0,\infty]$-ideal  is just  a contravariant $[0,\infty]$-functor $\phi\colon (X,d)\lra ([0,\infty],d_L)$ that transforms finite weighted colimits into finite weighted limits.  \end{rem}

Let $f\colon (X,d_X)\lra(Y,d_Y)$ be a non-expansive map.  It is  readily verified that if $\phi$ is a $[0,\infty]$-ideal    of $(X,d_X)$,  then $f^\ra(\phi)$ is a $[0,\infty]$-ideal  of $(Y,d_Y)$. We say that   $f$ is  \emph{$[0,\infty]$-Scott continuous} if $f$ preserves colimits of $[0,\infty]$-ideals \cite{AW2011}; that is, for each $[0,\infty]$-ideal $\phi$ of $(X,d_X)$, $f(\colim\phi)=\colim f^\ra(\phi)$ whenever $\colim\phi$ exists.

In the literature, $[0,\infty]$-Scott continuous maps are also known as \emph{Yoneda continuous} maps. Here is the reason.
An element $x$ of a metric space $(X,d)$ is called a \emph{Yoneda limit} of a forward Cauchy net $\{x_i\}_i$ if for all $y\in X$,
$$d(x,y)=\inf_i\sup_{j\geq i}d(x_j,y).$$
It is known (see e.g. \cite{FSW}) that   $x$ is a Yoneda limit of a forward Cauchy net $\{x_i\}_i$ if and only if $x$ is a colimit of the $[0,\infty]$-ideal $\inf_i\sup_{j\geq i}d(-,x_j)$. So, a non-expansive map preserves Yoneda limits (hence  Yoneda continuous) if and only if it is $[0,\infty]$-Scott continuous.

From Proposition \ref{colimit as join for ideal} one immediately derives the following:
\begin{prop}\label{dg}{\rm(\cite{LZ2020})} Let $(X,d_X)$ and $(Y,d_Y)$ be   cocomplete metric spaces. Then a non-expansive  map $f\colon (X,d_X)\lra(Y,d_Y)$ is $[0,\infty]$-Scott continuous if and only if, as an order-preserving map from $(X,\sqsubseteq_{d_X})$ to $(Y,\sqsubseteq_{d_Y})$, $f$ is Scott continuous. \end{prop}

\section{Compact functions of approach spaces}
While metric spaces are $[0,\infty]$-valued ordered sets, approach spaces are $[0,\infty]$-valued topological spaces. \begin{defn}(\cite{RL89,Lowen15}) An approach space is a set $X$ together with a map $\delta\colon X\times 2^X\lra[0,\infty]$, subject to the following conditions:  for all $x\in X$ and $A,B\in 2^X$,
\begin{enumerate}[label=\rm(A\arabic*)] \setlength{\itemsep}{0pt}
  \item  $\delta(x,\{x\})$=0;
  \item  $\delta(x,\varnothing)=\infty$;
  \item  $\delta(x,A\cup B)=\min\{\delta(x,A),\delta(x,B)\}$;
  \item  $\delta(x,A)\leq\sup\limits_{b\in B}\delta(b,A)+\delta(x,B)$.
\end{enumerate}
The value $\delta(x,A)$, called the distance from $x$ to $A$, measures how far $x$ is to $A$.

A   map  $f\colon (X,\delta_X)\lra (Y,\delta_Y)$ between approach spaces is continuous if
$$\delta_X(x,A)\geq\delta_Y(f(x),f(A))$$
for all $A\subseteq X$ and $x\in X$.\end{defn} Approach spaces and continuous maps form a category $$\sf App.$$

\begin{example}\label{p} (\cite{RL97,Lowen15})
For all $x\in[0,\infty]$ and $A\subseteq[0,\infty]$, let
$$\delta_{\mathbb{P}}(x,A)=
\begin{cases}
x\ominus\sup A, & A\neq\varnothing,\\
\infty, & A=\varnothing.
\end{cases}
$$
Then $\delta_{\mathbb{P}}$ is a distance on $[0,\infty]$. The space $$([0,\infty],\delta_{\mathbb{P}})$$ is denoted by $\mathbb{P}$.
\end{example}

Like topological spaces, approach spaces can be  described in many equivalent ways. Among  them we need the characterizations by   \emph{lower regular functions} and \emph{upper regular functions}. While the distance function  $\delta$ of an approach space $(X,\delta)$ is an analog of the closure operator of a topological space, lower regular functions  and  upper regular functions are, respectively, analogs of closed sets and open sets in the $[0,\infty]$-enriched context.
\begin{center}\begin{tabular}{  l l }
\toprule topological space ~~~ & approach space \\
\midrule
   closure operator   & distance map\\ \midrule   closed sets   & lower regular functions
     \\ \midrule  open sets & upper regular functions \\
\bottomrule \end{tabular}
\end{center}

\begin{defn}(\cite{RL97,Lowen15})
A collection $\mathfrak{L}\subseteq[0,\infty]^X$ is called a lower regular function frame if it satisfies:
\begin{enumerate}[label=\rm(L\arabic*)] \setlength{\itemsep}{0pt}
  \item[(L1)] For each subset $\{\phi_i\}_{i\in I}$ of $\mathfrak{L}$, $\sup_{i\in I}\phi_i\in\mathfrak{L}$.
  \item[(L2)] For all $\phi,\psi\in\mathfrak{L}$, $\inf\{\phi,\psi\}\in\mathfrak{L}X$.
  \item[(L3)] For all $\phi\in\mathfrak{L}$ and $r\in[0,\infty]$, both $\phi+r$ and $\phi\ominus r$ are in $\mathfrak{L}$.
\end{enumerate}
\end{defn}

For each approach space $(X,\delta)$, the set $\mathfrak{L}X$ of all  continuous maps $(X,\delta)\lra\mathbb{P}$ is a lower regular function frame, where $\mathbb{P}$ is the approach space given in Example \ref{p}.  Elements of $\mathfrak{L}X$  are called \emph{lower regular functions} of $(X,\delta)$. Conversely, if $\mathfrak{L}\subseteq[0,\infty]^X$ is a lower regular function frame, then $X$ together with the map $$\delta\colon X\times 2^X\lra [0,\infty], \quad \delta(x,A)=\sup\{\phi(x)\mid \phi\in\mathfrak{L}, \forall a\in A,\phi(a)=0\}$$
is an approach space with $\mathfrak{L}$ being its set of lower regular functions. Thus, every approach space is determined by its  lower regular functions, see e.g. \cite{RL97,Lowen15}.

 Let $(X,\delta)$ be an approach space and $x\in X$. It is readily verified that $\delta(-,\{x\})$ is a lower regular function and
\begin{equation}\label{YonedaforLR}
\rho_X(\delta(-,\{x\}),\phi)=\phi(x)
\end{equation}
for each lower regular function $\phi$ of $(X,\delta)$.

For each approach space $(X,\delta)$, let $\mathfrak{U}X$ be the subset of  $[0,\infty]^X$ consisting of elements $\lambda$ such that either $\lambda= \infty_X$ or $\lambda$ is bounded and $r\ominus\lambda\in\mathfrak{L}X$ for all $r\in[0,\infty]$.

It is not hard to check that $$\mathfrak{U}X=\{\infty_X\}\cup\Big\{\lambda\in[0,\infty]^X\mid \sup_{x\in X}\lambda(x)<\infty~\text{and}~\Big(\sup_{x\in X}\lambda(x)\Big)\ominus\lambda\in\mathfrak{L}X \Big\} $$ and that $\mathfrak{U}X$ satisfies the following conditions:\begin{enumerate}[label=\rm(U\arabic*)] \setlength{\itemsep}{0pt}
  \item  $\infty_X\in\mathfrak{U}X$.
  \item For all $\lambda\in\mathfrak{U}X$, if  $\lambda\not=\infty_X$, then $\lambda$ is bounded.
  \item  For each subset $\{\lambda_i\}_{i\in I}$ of $\mathfrak{U}X$, $\inf_{i \in I}\lambda_i \in\mathfrak{U}X$.
  \item  For all $\lambda,\mu\in\mathfrak{U}X$, $\sup\{\lambda,\mu\} \in\mathfrak{U}X$.
  \item  For all $\lambda\in\mathfrak{U}X$ and $r\in[0,\infty]$, both $\lambda+r$ and $\lambda\ominus r$ are in $\mathfrak{U}X$.
\end{enumerate} Elements of $\mathfrak{U}X$ are called \emph{upper regular functions} of $(X,\delta)$ \cite{Lowen15}. We hasten to note that the definition given here is slightly different from that in \cite{Lowen15}, the difference is that   the function $\infty_X$, which is an analog of the empty set, is postulated to be upper regular.
 It is known  (see e.g. \cite{Lowen15}) that $\mathfrak{L}X$ is the smallest lower regular function frame that contains $$\{r\ominus\lambda\mid r\in[0,\infty],\lambda\in\mathfrak{U}X\}.$$ Thus, every approach space is also characterized by its upper regular functions.

\begin{prop}{\rm(\cite[Theorem 1.3.3]{Lowen15})} For a map $f\colon(X,\delta_X)\lra(Y,\delta_Y)$ between approach spaces, the following are equivalent: \begin{enumerate}[label=\rm(\arabic*)] \setlength{\itemsep}{0pt}
  \item $f$ is a continuous map. \item For each lower regular function $\phi$ of $(Y,\delta_Y)$, $\phi\circ f$ is a lower regular function  of $(X,\delta_X)$. \item For each upper regular function $\lambda$ of $(Y,\delta_Y)$, $\lambda\circ f$ is an upper regular function  of $(X,\delta_X)$.\end{enumerate}\end{prop}

\begin{prop}
For each approach space $(X,\delta)$, both $(\mathfrak{L}X,\rho_X)$ and $(\mathfrak{U}X,\rho_X)$ are cocomplete metric spaces.\end{prop}

\begin{proof}
It follows from the fact that $\mathfrak{L}X$ is closed in $([0,\infty]^X,\rho_X)$ under the formation of weighted limits and $\mathfrak{U}X$ is closed in $([0,\infty]^X,\rho_X)$ under the formation of weighted colimits.
\end{proof}

In particular, for each $\phi\in\mathfrak{L}X$ and $r\in[0,\infty]$, $r+\phi $ and $\phi\ominus r$ are the tensor and cotensor of $\phi$ with $r$ in $(\mathfrak{L}X,\rho_X)$, respectively; for each $\lambda\in\mathfrak{U}X$ and $r\in[0,\infty]$, $\lambda+r$ and $\lambda\ominus r$ are the tensor and cotensor of $\lambda$ with $r$ in $(\mathfrak{U}X,\rho_X)$, respectively.

For each approach space $(X,\delta)$, it is clear that $$d_\delta(x,y)\coloneqq\delta(x,\{y\})$$ is a metric on $X$, called  the \emph{specialization metric} of $(X,\delta)$. Assigning to each approach space its specialization metric gives rise to a functor
$$\Omega\colon {\sf App}\lra {\sf Met}.$$
Every lower regular function of $(X,\delta)$ is a weight and every upper regular function is a coweight of $(X,d_\delta)$.

A function $\lambda\colon X\lra[0,\infty]$ is  \emph{saturated}   if it is a coweight of the metric space $(X,d_\delta)$; that is,  $d_\delta(x,y)+\lambda(x)\geq \lambda(y)$ for all $x,y\in X$. It is easily verified that
\begin{equation}\label{YonedaforRL}
\delta(-,\{x\})\cpt\lambda=\lambda(x)
\end{equation}
for each saturated function $\lambda$ and each $x\in X$.

All upper regular functions are saturated; we have more.

\begin{prop}
Let $(X,\delta)$ be an approach space and $\theta\colon X\lra[0,\infty]$. \begin{enumerate}[label=\rm(\roman*)] \setlength{\itemsep}{0pt}
  \item The function $ \uparrow\!\theta\coloneqq\sup\{\lambda\in\mathfrak{U}X\mid\lambda \leq\theta\}$ is   saturated. \item  $\theta$ is saturated if and only if $\theta = \uparrow\!\theta$.\end{enumerate} Thus, $\uparrow\!\theta$ is called the saturation of $\theta$.
\end{prop}

\begin{proof}
(i) Trivial, since every upper regular function of $(X,\delta)$ is a coweight of $(X,d_\delta)$.

(ii)
It suffices to check that if  $\theta$ is  saturated   then $\theta\leq\uparrow\!\theta$. Let $x\in X$. If $\theta(x)<\infty$, then $\theta(x)\ominus\delta(-,\{x\})$ is  upper regular   and $\theta(x)\ominus\delta(-,\{x\})\leq\theta$. So, $\theta(x)=\theta(x)\ominus\delta(x,\{x\})\leq\uparrow\!\theta(x)$. If $\theta(x)=\infty$, then for each $r<\infty$, $r\ominus\delta(-,\{x\})$ is upper regular and $r\ominus\delta(-,\{x\})\leq\theta$. So $\uparrow\!\theta(x)\geq r$, and consequently, $\uparrow\!\theta(x)=\infty$ by arbitrariness of $r$.
\end{proof}

\begin{cor}\label{saturation}
Let $(X,\delta)$ be an approach space and $\theta\in[0,\infty]^X$. \begin{enumerate}[label=\rm(\roman*)] \setlength{\itemsep}{0pt}
  \item For each upper regular function $\lambda$ of $X$,
$\rho_X(\theta,\lambda)=\rho_X(\uparrow\!\theta,\lambda).$ \item
$\uparrow\!\theta=\sup\{\lambda\ominus\rho_X(\theta,\lambda)\mid \lambda\in\mathfrak{U}X\}.$\end{enumerate}
\end{cor}

Let $(X,\delta)$ be an approach space and let  $\theta\in[0,\infty]^X$. By Proposition \ref{real-valued filter} it is seen readily that
$$\rho_X(\theta,-)\colon \mathfrak{U}X\lra[0,\infty]$$ is a $[0,\infty]$-filter of the cocomplete metric space $(\mathfrak{U}X, \rho_X)$.

\begin{defn} (\cite{AW2011}) A $[0,\infty]$-filter $\lambda$ of a metric space $(X,d)$ is  open  if  $\lambda\colon(X,d)\lra([0,\infty],d_L)$
is $[0,\infty]$-Scott continuous.\end{defn}

Now we introduce the central notion of this paper.

\begin{defn}\label{cf}
Let $(X,\delta)$ be an approach space. A function $\theta\colon X\lra [0,\infty]$ is  compact if $\rho_X(\theta,-)$ is an open $[0,\infty]$-filter of the metric space $(\mathfrak{U}X,\rho_X)$ of upper regular functions.
\end{defn}
This definition is a direct extension of the characterization of compact subsets in \cite[Corollary 2.14]{HM1981} to the $[0,\infty]$-enriched context.

The following property of compact functions follows immediately from Corollary \ref{saturation}.
\begin{prop}\label{upcom}
Let $(X,\delta)$ be an approach space and $\theta\colon X\lra [0,\infty]$ be a function. Then $\theta$ is compact if and only if so is its saturation $\uparrow\!\theta$.
\end{prop}

Since   $(\mathfrak{U}X,\rho_X)$ is a cocomplete metric space and its underlying order is opposite to the usual pointwise order, from Proposition \ref{dg} we immediately derive the following characterization of compact functions.

\begin{prop}
Let $(X,\delta)$ be a metric space and $\theta\colon X\lra [0,\infty]$. Then $\theta$ is a compact function of $(X,\delta)$ if and only if for each filtered family $\{\lambda_i\}_{i\in I}$ of upper regular functions,
$$\rho_X(\theta,\inf_{i\in I}\lambda_i)=\inf_{i\in I}\rho_X(\theta,\lambda_i).$$
\end{prop}

Parallel to the fact that continuous images of compact sets are compact, we have:
\begin{prop}Let $f\colon(X,\delta_X)\lra(Y,\delta_Y)$ be a continuous map between approach spaces. Then for each compact function $\theta$   of $(X,\delta_X)$, the function $$f(\theta)\colon Y\lra[0,\infty],\quad f(\theta)(y)=\inf\{\theta(x)\mid f(x)=y\}$$ is  compact in $(Y,\delta_Y)$.\end{prop}

\begin{proof}This follows directly from the above proposition and the fact that  \[\rho_Y(f(\theta),\lambda)=\rho_X(\theta,\lambda\circ f) \] for every upper regular function $\lambda$ of $(Y,\delta_Y)$. \end{proof}

In a topological space $X$, a subset $K$ is  compact if and only if for each filtered family of closed sets $\{F_i\}_{i\in I}$ such that $K$ meets each $ F_i$, $K$ meets the intersection $\bigcap_{i\in I}F_i$. The following proposition is an analog of this fact in the enriched context.
\begin{prop}\label{compactness via lower regular}
Let $(X,\delta)$ be an approach space. For each function  $\theta\colon X\lra [0,\infty]$, the following are equivalent: \begin{enumerate}[label=\rm(\arabic*)] \setlength{\itemsep}{0pt}
  \item $\theta$ is compact.  \item  For each directed family $\{\phi_j\}_{j\in J}$ of lower regular functions,
  $$(\sup_{j\in J}\phi_j)\cpt\theta=\sup_{j\in J} (\phi_j\cpt\theta).$$
\item The map \[-\pitchfork\theta\colon (\mathfrak{L}X,\rho_X)\lra([0,\infty],d_L) \] preserves limits of $[0,\infty]$-filters. \end{enumerate}
\end{prop}

\begin{proof}
$(1)\Rightarrow(2)$
We proceed with two cases.

If $\sup_{j\in J}\phi_j$ is bounded by $b<\infty$, then $\{b\ominus\phi_j\}_{j\in J}$ is a filtered family of upper regular functions of $(X,\delta)$. Appealing to Lemma \ref{sub vs tensor}  we have
\begin{align*}
b\ominus((\sup_{j\in J}\phi_j)\cpt\theta) & = \rho_X(\theta,b\ominus\sup_{j\in J}\phi_j)\\
& = \rho_X(\theta,\inf_{j\in J}(b\ominus\phi_j))\\
& = \inf_{j\in J}\rho_X(\theta,b\ominus\phi_j)\\
& = \inf_{j\in J}b\ominus(\phi_j\cpt\theta)\\
& = b\ominus \sup_{j\in J} (\phi_j\cpt\theta) .
\end{align*}
From the arbitrariness of $b$ it follows that $(\sup_{j\in J}\phi_j)\cpt\theta=\sup_{j\in J} (\phi_j\cpt\theta)$.

If $\sup_{j\in I}\phi_j$ is not bounded, then
\begin{align*}
(\sup_{j\in J}\phi_j)\cpt\theta & = (\sup_{r\in[0,\infty)}\inf\{r_X,\sup_{j\in J}\phi_j\})\cpt\theta\\
& = \sup_{r\in[0,\infty)} (\inf\{r_X,\sup_{j\in J}\phi_j\}\cpt\theta)\\
& = \sup_{r\in[0,\infty)}((\sup_{j\in J}\inf\{r_X,\phi_j\})\cpt\theta)\\
& = \sup_{r\in[0,\infty)}\sup_{j\in J}(\inf\{r_X,\phi_j\} \cpt\theta)\\
& = \sup_{j\in J}\sup_{r\in[0,\infty)} (\inf\{r_X,\phi_j\} \cpt\theta)\\
& = \sup_{j\in J}(\phi_j\cpt\theta).
\end{align*}

$(2)\Rightarrow(1)$ Let  $\{\lambda_i\}_{i\in I}$ be a filtered family of upper regular functions. Without loss of generality, we may suppose that all $\lambda_i$ are bounded by $b<\infty$. Then $\{b\ominus \lambda_i\}_{i\in I}$ is a directed family of lower regular functions. Hence by Lemma \ref{sub vs tensor} we have
\begin{align*}
\rho_X(\theta,\inf_{i\in I}\lambda_i) & = b\ominus ((b\ominus\inf_{i\in I}\lambda_i)\cpt\theta)\\
& =  b\ominus ((\sup_{i\in I}(b\ominus \lambda_i))\cpt\theta)\\
& =  b\ominus  \sup_{i\in I}((b\ominus \lambda_i)\cpt\theta)\\
& = \inf_{i\in I} b\ominus( (b\ominus\lambda_i)\cpt\theta)\\
& = \inf_{i\in I} \rho_X(\theta,\lambda_i),
\end{align*}
which shows that $\theta$ is compact.

$(2)\Leftrightarrow(3)$ This follows  from the dual of Proposition \ref{colimit as join for ideal}. \end{proof}

\begin{prop}\label{compact function of subspace} Let $(X,\delta)$ be an approach space and $A\subseteq X$. Then a function $\theta\colon A\lra[0,\infty]$ is  compact in the subspace $(A,\delta|A)$ if and only if \[\theta^*\colon X\lra[0,\infty],\quad \theta^*(x)=\begin{cases}\theta(x), &x\in A \\ \infty, &x\notin A\end{cases}\] is compact in $(X,\delta)$. \end{prop}

\begin{proof}This follows from Proposition \ref{compactness via lower regular} and the fact that   $\lambda\colon A\lra[0,\infty]$ is a lower regular function of $(A,\delta|A)$ if and only if there is some lower regular function $\widetilde{\lambda}$ of $(X,\delta)$ such that $\lambda$ is the restriction of $\widetilde{\lambda}$ on $A$. \end{proof}

The following proposition implies that the family of compact functions is determined by the inhabited ones. We note that for each approach space $(X,\delta)$, the function $\infty_X$  is compact.

\begin{prop}\label{finunion}
Let $(X,\delta)$ be an approach space. 
\begin{enumerate}[label=\rm(\roman*)] \setlength{\itemsep}{0pt}
  \item If both $\theta_1$ and $\theta_2$  are compact functions of $(X,\delta)$, then so is $\inf\{\theta_1,\theta_2\}$.
  \item If $\theta$ is compact then so is $\theta+r$  for each $r\in[0,\infty]$.
  \item If $\theta$ is compact and $\theta\neq\infty_X$, then for each $r\leq\inf\limits_{x\in X}\theta(x)$, $\theta\ominus r$ is compact.
\end{enumerate}

Hence for each   $\theta\neq\infty_X$,   $\theta$ is compact if and only if   $\theta\ominus\inf_{x\in X}\theta(x)$ is compact.
\end{prop}

\begin{proof}
(i) Use the fact that for any upper regular function $\lambda$ of  $(X,\delta)$, $\rho_X(\inf\{\theta_1,\theta_2\},\lambda)=\max\{\rho_X(\theta_1 ,\lambda), \rho_X(\theta_2,\lambda)\}$.

(ii) For each directed family $\{\phi_i\}_{i\in I}$ of lower regular functions, we have
\begin{align*}
 (\sup_{i\in I}\phi_i)\cpt(\theta+r) & = \inf_{x\in X}(\sup_{i\in I}\phi_i(x)+\theta(x)+r)\\
& = (\sup_{i\in I}\phi_i)\cpt \theta +r \\
& =  \sup_{i\in I}(\phi_i \cpt \theta)+r\\
& = \sup_{i\in I}(\phi_i \cpt(\theta+r)),
\end{align*} hence $\theta+r$ is compact.

(iii) Similar.\end{proof}

\begin{rem}The conclusions (i) and (ii) in Proposition \ref{finunion}  amount to that  compact functions of an approach space $(X,\delta)$ are closed in $([0,\infty]^X,\rho_X)$ under formation of finite weighted colimits (c.f. Remark \ref{finite weighted limits}), which echoes the fact that a finite union of compact subsets is compact.
\end{rem}

\begin{prop}\label{appbycom}
Let $(X,\delta)$ be an approach space and $\theta\colon X\lra [0,\infty]$ be a function. If for each $r>0$, there exists a compact function $\xi$  such that $\theta\leq \xi \leq \theta+r$, then $\theta$ is   compact.
\end{prop}

\begin{proof}
Let $\{\phi_i\}_{i\in I}$ be a directed family of lower regular functions and let $r>0$. Pick a compact function $\xi$   such that $\theta\leq \xi \leq \theta+r$. Then
\begin{align*}
(\sup_{i\in I}\phi_i)\cpt\theta & \leq (\sup_{i\in I}\phi_i)\cpt\xi\\
& =  \sup_{i\in I}(\phi_i \cpt\xi)\\
& \leq \sup_{i\in I}(\phi_i \cpt(\theta+r))\\
& =  \sup_{i\in I}(\phi_i \cpt\theta)+r,
\end{align*}
 hence $(\sup_{i\in I}\phi_i)\cpt\theta\leq\sup_{i\in I}(\phi_i \cpt\theta)$ by arbitrariness of $r$. It follows that $\theta$ is compact since the converse inequality is obvious.
\end{proof}

\begin{example} Every function $\theta\colon[0,\infty]\lra[0,\infty]$ is a compact function of $\mathbb{P}$. By Proposition \ref{upcom} and Proposition \ref{finunion}, we only need to check that every inhabited and saturated function of $\mathbb{P}$ is compact. For each $a\in[0,\infty]$ and $r\in[0,\infty]$, define a function $\theta_{a,r}\colon[0,\infty]\lra[0,\infty]$ by
$$\theta_{a,r}(x)=\max\{a\ominus x,r\}.$$
Note that if $r\geq a$, then $\theta_{a,r}(x)=r$ for all $x\in[0,\infty]$. For each lower regular function $\phi$ of $\mathbb{P}$, since $\phi$ is a weight of the metric space $([0,\infty],d_R)$, we have
$$\phi\cpt\theta_{a,r}=\phi(a\ominus r)+r.$$
Then, with the help of Proposition \ref{compactness via lower regular}, one   verifies readily  that $\theta_{a,r}$ is compact. Suppose that $\lambda$ is an inhabited and saturated function of $\mathbb{P}$. Then $\lambda$ satisfies the following conditions:
\begin{enumerate}[label=\rm(\roman*)] \setlength{\itemsep}{0pt}
  \item $\lambda$ is decreasing and continuous at $[0,\infty)$;
  \item If $y\leq x$, then $\lambda(y)-\lambda(x)\leq x-y$;
  \item $\lambda(\infty)=0$.
\end{enumerate}
If $\lambda(0)=\infty$, then $\lambda=\theta_{\infty,0}$ by (ii) and (iii), hence compact. If $\lambda(0)\neq\infty$, without loss of generality, we assume that $\lambda(0)=1$. By Proposition \ref{appbycom}, it suffices to show that for each $n\geq 1$, there exists a compact function $\mu$ such that $\lambda\leq\mu\leq\lambda+\frac{1}{n}$. Let $a_0=\infty$ and
$$a_i=\inf\{x\in[0,\infty]\mid \lambda(x)\leq \frac{i}{n}\}$$
for each $1\leq i\leq n$. Then $\lambda\leq \theta_{a_i+\frac{i}{n},\frac{i}{n}}$ and for any $x\in [a_i,a_{i-1}]$, we have
$$ \theta_{a_i+\frac{i}{n},\frac{i}{n}}(x)\leq\lambda(x)+\frac{1}{n},$$
hence
$$\mu\coloneqq\inf\{\theta_{a_i+\frac{i}{n},\frac{i}{n}}\mid 1\leq i\leq n\}$$ satisfies the requirement.
\end{example}

In \cite{Lowen88}, Lowen introduced the notion of \emph{index of compactness} for approach spaces. To see the relationship between compact functions and the compactness-indices of  approach spaces, we introduce the following:

\begin{defn}\label{compact index of func} Let $(X,\delta)$ be an approach space and $\theta\in[0,\infty]^X$. The value \[\chi(\theta)\coloneqq \sup_{I}\Big((\sup_{\phi\in I}\phi)\pitchfork\theta\ominus\sup_{\phi\in I}(\phi\pitchfork\theta)\Big)\] is called the  index of compactness of $\theta$, where $I$ ranges over all ideals of the complete lattice $\mathfrak{L}X$ of lower regular functions of $(X,\delta)$.\end{defn}
We leave it to the reader to check that  \[\chi(\theta)= \sup_{F}\Big(\Big(\inf_{\lambda\in F} \rho_X(\theta,\lambda) \Big)\ominus\rho_X(\theta,\inf_{\lambda\in  F}\lambda)\Big),\]   where $F$ ranges over all filters of the complete lattice $\mathfrak{U}X$ of upper regular functions of $(X,\delta)$.

It is clear that $\theta$ is a compact function if and only if $\chi(\theta)=0$; that is, $\theta$ has $0$ as index of compactness.

\begin{prop}Let $(X,\delta)$ be an approach space.
\begin{enumerate}[label=\rm(\roman*)] \setlength{\itemsep}{0pt}
  \item The index of compactness of the constant function $0_X$ is given by \[\chi(0_X)=   \sup_J\inf_{x\in X}\sup_{\phi\in J}\phi(x) ,\]   where $J$ ranges over all  ideals composed of inhabited lower regular functions.
  \item $\chi(0_X)$ is equal to the index of compactness of $(X,\delta)$  in the sense of Lowen {\rm\cite{Lowen88,Lowen15}}. \item $\chi(0_X)=0$ if and only if $(X,\delta)$ is $0$-compact in the sense of Lowen {\rm\cite{Lowen88,Lowen15}}. \end{enumerate}\end{prop}
\begin{proof}(i) By definition, \[\chi(0_X)= \sup_{I}\Big(\inf_{x\in X}\sup_{\phi\in I}\phi(x)\ominus\sup_{\phi\in I}\inf_{x\in X}\phi(x)\Big),\]where $I$ ranges over all ideals of the complete lattice $\mathfrak{L}X$. Since $\sup_{\phi\in J}\inf_{x\in X}\phi(x)=0$ for each ideal $J$ composed of inhabited lower regular functions, it suffices to check that for each ideal $I$ of $\mathfrak{L}X$, there is an ideal $J$   composed of inhabited lower regular functions such that \[\inf_{x\in X}\sup_{\phi\in I}\phi(x)\ominus\sup_{\phi\in I}\inf_{x\in X}\phi(x)\leq  \inf_{x\in X}\sup_{\psi\in J}\psi(x).\]

If $\sup_{\phi\in I}\inf_{x\in X}\phi(x)=\infty$,  any ideal composed of inhabited lower regular functions meets the requirement.  If $\sup_{\phi\in I}\inf_{x\in X}\phi(x)=b<\infty$, then  $\{\phi\ominus b\}_{\phi\in I}$ is a directed family consisting of inhabited lower regular functions. Let $J$ be the ideal generated by this family. Then $J$ is composed of inhabited lower regular functions and
$$\inf_{x\in X}\sup_{\psi\in J}\psi(x)= \inf_{x\in X}\sup_{\phi\in I}(\phi(x)\ominus b)= \inf_{x\in X}\sup_{\phi\in I}\phi(x)\ominus b,$$
which shows that $J$ satisfies the requirement.

(ii) Follows from (i) and \cite[Theorem 4.3.9]{Lowen15}.

(iii) A special case of (ii). \end{proof}

The above proposition shows that Definition \ref{compact index of func} provides an extension of the notion of index of compactness for approach spaces in the sense of Lowen; and it also leads to the following:
\begin{defn}
An approach space $(X,\delta)$ is compact if its index of compactness is $0$; or equivalently,    $0_X$ is a compact function. \end{defn}

Compact approach spaces are just the $0$-compact ones in the sense of Lowen \cite{Lowen15}.
It follows from Lemma \ref{finunion} that an approach space $(X,\delta)$ is compact if and only if for all $r\in[0,\infty]$, the constant function $r_X$ is  compact.

In the following we show that   compact functions in approach spaces are an  \emph{extension}  of   compact subsets in topological spaces and totally bounded subsets in metric spaces.

For each topological space $X$, the map $$\delta_X\colon X\times 2^X\lra [0,\infty],\quad
\delta_X(x,A)=
\begin{cases}
0, & \text{if $x$ is in the closure of $A$},\\
\infty, & \text{otherwise}
\end{cases}
$$
is  an approach structure on the set $X$. The correspondence $X \mapsto \omega(X)\coloneqq(X,\delta_X)$ defines a full and faithful functor
$$\omega\colon \sf Top\lra \sf App.$$
Approach spaces of the form $\omega(X)$ are said to be \emph{topologically generated} \cite{RL97,Lowen15}. It is clear that a function $\phi\colon X\lra [0,\infty]$ is a lower regular function of $\omega(X)$ if and only if $\phi$ is lower semicontinuous; and that  $\lambda\colon X\lra [0,\infty]$ is an upper regular function of $\omega(X)$ if and only if either $\lambda=\infty_X$ or $\lambda$ is bounded and upper semicontinuous.

For a topological space $X$, it is shown in \cite[Proposition 4.3.11]{Lowen15} that the approach space $\omega(X)$ is compact if and only if the topological space $X$  is compact. This fact together with Proposition \ref{compact function of subspace} yields the following proposition, which shows that the notion of compact functions of approach spaces extends that of  compact subsets of topological spaces.

\begin{prop}
Let $X$ be a topological space and $A$ a subset of $X$. Then $0_A$ is a compact function of $\omega(X)$ if and only if $A$ is a compact subset of $X$. \end{prop}



\begin{prop}
Let $X$ be a topological space and let $\theta\colon X\lra[0,\infty]$ be a lower semicontinuous function. Then $\theta$ is a compact function of $\omega(X)$ if and only if $\theta^{-1}([0,r])$ is a compact subset of $X$ for each $r\in[0,\infty)$.
\end{prop}

\begin{proof}
We check the necessity first. Let $A=\theta^{-1}[0,r]$. Since $\theta$ is lower semicontinuous, $A$ is a closed set in X. Let $\{F_i\}_{i\in I}$ be a filtered family of non-empty closed sets of $A$. Since $A$ is   closed, it follows that $\{0_{F_i}\}_{i\in I}$ is a directed family of lower regular functions of $\omega(X)$, hence $$(\sup_{i\in I}0_{F_{i}})\cpt \theta=\sup_{i\in I} (0_{F_i}\cpt \theta) \leq r,$$ which shows that $\{F_i\}_{i\in I}$ has a nonempty intersection and consequently, $A$ is compact.

As for sufficiency, it suffices to prove that, for each directed family  of lower semicontinuous functions  $\{\phi_i\colon X\lra [0,\infty]\}_{i\in I}$,
$$ (\sup_{i\in I}\phi_i)\cpt \theta\leq\sup_{i\in I}(\phi_i\cpt \theta).$$
If $\sup_{i\in I}(\phi_i\cpt\theta)<r<\infty$, then for each $i\in I$,
$$\inf_{x\in X}(\phi_i(x)+\theta(x))<r.$$
Let $F_i=(\phi_i+\theta)^{-1}[0,r]$ for each $i\in I$. By lower semicontinuity of $\theta$ and $\phi_i$, one sees that  $\{F_i\}_{i\in I}$ is a filtered family of non-empty closed sets contained in the closed set $A\coloneqq \theta^{-1}[0,r]$. Since $A$ is compact, then
$$(\sup_{i\in I}\phi_i)\cpt \theta=\inf_{x\in X}(\sup_{i\in I}\phi_i(x)+\theta(x))\leq r$$
and consequently,
$$(\sup_{i\in I}\phi_i)\cpt \theta\leq\sup_{i\in I}(\phi_i\cpt \theta)$$
as desired.
\end{proof}
\begin{cor}
 If $X$ is a compact topological space, then every lower semicontinuous function from  $X$ to $[0,\infty]$ is compact in $\omega(X)$.
\end{cor}

Given a  metric space $(X,d)$, the map  $$\Gamma(d)\colon X\times 2^X\lra[0,\infty],\quad
\Gamma(d)(x,A)=
\begin{cases}
\infty, & A=\varnothing,\\
\inf\limits_{y\in A}d(x,y), & A\neq\varnothing
\end{cases}
$$ is an approach structure on the set $X$. Actually, the correspondence $(X,d)\mapsto (X,\Gamma(d))$ defines a functor
$$\Gamma: \sf Met\lra \sf App,$$
which is left adjoint to $\Omega: \sf App\lra \sf Met$.
Spaces of the form $(X,\Gamma(d))$ are called \emph{metric approach spaces} \cite{RL97,Lowen15}.

It is not hard to check that $\phi\colon X\lra [0,\infty]$ is a lower regular function of the metric approach space $(X,\Gamma(d))$ if and only if $\phi$ is a weight of $(X,d)$;  and  $\lambda\colon X\lra [0,\infty]$ is an upper regular function of $(X,\Gamma(d))$ if and only if either $\lambda=\infty_X$ or $\lambda$ is a bounded coweight of $(X,d)$.
Then, by the dual statement of the Yoneda lemma (Lemma \ref{yoneda lemma}) one sees that for each $x\in X$,   $d(x,-)\colon X\lra[0,\infty]$ is a compact function of $(X,\Gamma(d))$. So, by Proposition \ref{finunion} (i), for each finite subset $F$ of   $X$, $$d(F,-)\coloneqq\inf_{a\in F}d(a,-)$$ is a compact function of $(X,\Gamma(d))$.

A metric space $(X,d)$ is \emph{precompact} if for each $r>0$, there exists a finite subset $A$ of $X$, such that
$X=\bigcup_{x\in A}B(x,r),$
where $B(x,r)=\{y\in X\mid d(x,y)<r\}$. A subset $K$ is called a precompact subset   if $(K,d)$ is a precompact metric space.

\begin{prop}\label{totally bounded subset}
Let $(X,d)$ be a metric space and $K$ a subset of $X$. The following are equivalent: \begin{enumerate}[label=\rm(\arabic*)] \setlength{\itemsep}{0pt}
  \item $K$ is a precompact subset of $(X,d)$.
  \item  $d(K,-)$ is a compact function of $(X,\Gamma(d))$.
  \item  $0_K$ is a compact function of $(X,\Gamma(d))$.
\end{enumerate}
\end{prop}

\begin{proof}

$(1) \Rightarrow  (2)$ Since $K$ is a precompact subset of $X$, for each $r>0$  there exists a finite subset $F$ of $K$ such that, $$K\subseteq\bigcup_{x\in F}B(x,r).$$
It follows that
$$d(K,y)\leq d(F,y)<d(K,y)+r$$
for each $y\in X$. Since $F$ is finite, $d(F,-)$ is a compact function, so, $d(K,-)$ is compact  by Proposition \ref{appbycom}.

$(2)  \Rightarrow  (3)$ This follows from   that for each weight $\phi$ of $(X,d)$,
\begin{align*}
\phi\cpt d(K,-)  & =\inf_{x\in X}(\phi(x)+d(K,x))\\
& = \inf_{x\in K}\phi(x)\\
& =  \phi\cpt0_K.
\end{align*}

$(3) \Rightarrow  (1)$  Suppose on the contrary that $K$ is not precompact. Then there exists $r>0$ such that $K\nsubseteq\bigcup_{x\in F}B(x,r)$ for any finite subset $F$ of $K$. For each $x\in K$, let $$\phi_x=d(-,X\backslash B(x,r)).$$ Since $\phi_x(x)\geq r$ for all $x\in K$, it follows that
$$ (\sup_{x\in K}\phi_x)\cpt0_K=\inf_{y\in K}\sup_{x\in K}\phi_x(y)\geq r.$$

For each finite subset $F$ of $K$, there is some $z\in K$ such that $d(x,z)\geq r$ for all $x\in F$. Thus,
$$(\sup_{x\in F}\phi_x)\cpt0_K=\inf_{y\in K}\sup_{x\in F}\phi_x(y)\leq\sup_{x\in F}\phi_x(z)=0,$$
which contradicts the compactness of $0_K$.
\end{proof}

\section{A Hofmann-Mislove theorem}

\begin{defn}
Let $(X,\delta)$ be an approach space and let $\Lambda$ be  a $[0,\infty]$-filter  of the metric space $(\mathfrak{U}X,\rho_X)$   of upper regular functions of $(X,\delta)$. We say that   $\Lambda$ is
\begin{enumerate}[label=\rm(\roman*)] \setlength{\itemsep}{0pt}
  \item proper if $\Lambda(r_X)=r$ for each $r\in[0,\infty]$;
  \item decent if $\Lambda(\lambda)\neq 0$ for some   $\lambda\neq\infty_X$.
\end{enumerate}
\end{defn}

For each approach space $(X,\delta)$ with $X$ being non-empty, the metric space $(\mathfrak{U}X,\rho_X)$ has exactly two  $[0,\infty]$-filters that are not decent, both of them are open. The first is
$\rho_X(\infty_X,-)$ and the second is given by
\begin{equation*}
\Lambda(\lambda)=
\begin{cases}
0, & \lambda\neq\infty_X,\\
\infty, & \lambda=\infty_X.
\end{cases}
\end{equation*}

The following proposition is obvious.
\begin{prop}
Let $\theta$ be a compact saturated function of an approach space $(X,\delta)$. \begin{enumerate}[label=\rm(\roman*)] \setlength{\itemsep}{0pt} \item If $\theta$ is inhabited  then $\rho_X(\theta,-)$ is a  proper open $[0,\infty]$-filter of $(\mathfrak{U}X,\rho_X)$.
\item If  $\theta$  is non-vacuous  then $\rho_X(\theta,-)$ is a  decent open $[0,\infty]$-filter of $(\mathfrak{U}X,\rho_X)$.
\end{enumerate}
\end{prop}


Our aim is to show that for a sober approach space $(X,\delta)$, every proper open $[0,\infty]$-filter of $(\mathfrak{U}X,\rho_X)$ is of the form $\rho_X(\theta,-)$ for a unique inhabited and saturated compact function.

A lower regular function $\phi$ of an approach space  $(X,\delta)$ is said to be \emph{irreducible} if for all lower regular functions $\phi_1$ and $\phi_2$,    $$\inf\{\phi_1,\phi_2\}\leq\phi\implies \text{either $\phi_1\leq\phi$ or $\phi_2\leq\phi$}.$$

\begin{defn}(\cite{BRC})
An approach space $(X,\delta)$ is sober if for each inhabited and irreducible lower regular function $\phi$ of $(X,\delta)$, there is a unique $x\in X$ such that $\phi=\delta(-,\{x\})$.
\end{defn}

\begin{example}\begin{enumerate}[label=\rm(\alph*)] \setlength{\itemsep}{0pt}
  \item (\cite{BRC}) For each topological space $X$, the approach space $\omega(X)$ is   sober  if and only if $X$, as a topological space, is   sober. \item   (\cite{LiZ18a}) For each  metric space $(X,d)$, the metric approach space $(X,\Gamma(d))$ is   sober   if and only if $(X,d)$ is Smyth complete; that is,  $(X,d)$ is  separated and all of its forward Cauchy nets converge in its symmetrization (defined in an obvious way). In the case that the metric $d$ is symmetric, this conclusion was first proved in \cite{BRC}.\end{enumerate}\end{example}

\begin{prop}An approach space $(X,\delta)$ is sober if and only if for each non-vacuous irreducible lower regular function $\phi$ of $(X,\delta)$,  there exist a unique $x\in X$ and a unique $s\in[0,\infty)$  such that $\phi=s+\delta(-,\{x\})$.\end{prop}
\begin{proof}Consider the lower regular function $\phi\ominus\inf_{x\in X}\phi(x)$.\end{proof}

Now we state the main result of this paper.
\begin{thm}\label{hm}
Let $(X,\delta)$ be a sober approach space.
\begin{enumerate}[label=\rm(\roman*)] \setlength{\itemsep}{0pt}
  \item  For each   proper open $[0,\infty]$-filter $\Lambda$ of $(\mathfrak{U}X,\rho_X)$, there is a unique inhabited  compact saturated function $\nu$ of $(X,\delta)$ such that $\Lambda=\rho_X(\nu,-)$.
  \item  For each decent open $[0,\infty]$-filter $\Lambda$ of $(\mathfrak{U}X,\rho_X)$, there is a unique non-vacuous compact  saturated function $\mu$ of $(X,\delta)$ such that $\Lambda=\rho_X(\mu,-)$.
\end{enumerate}
Therefore, for a sober approach space $(X,\delta)$, there is a bijection between  proper (decent, resp.) open $[0,\infty]$-filters of the metric space of upper regular functions  of $(X,\delta)$ and   inhabited (non-vacuous, resp.) compact saturated functions of $(X,\delta)$. Furthermore, this bijection is an isomorphism between  metric spaces.
\end{thm}


\begin{lem}\label{HMC}
Let $(X,\delta)$ be an approach space. If $\lambda$ is a compact saturated function of $(X,\delta)$, then the map
$$\Phi \colon \mathfrak{L}X\lra[0,\infty], \quad \Phi(\phi)=\phi\cpt\lambda$$ satisfies the following conditions:
\begin{enumerate}[label=\rm(C\arabic*)] \setlength{\itemsep}{0pt}
  \item For each directed subset $\{\phi_i\}_i$ of $\mathfrak{L}X$, $\Phi(\sup\limits_{i\in I}\phi_i)=\sup\limits_{i\in I} \Phi(\phi_i)$.
  \item  For all $\phi_1,\phi_2\in\mathfrak{L}X$, $ \Phi(\inf\{\phi_1,\phi_2\})=\min\{ \Phi(\phi_1), \Phi(\phi_2)\}$.
  \item  For all $\phi\in\mathfrak{L}X$ and $r\in[0,\infty]$, $ \Phi(\phi+r) =  \Phi(\phi)+r$.
\end{enumerate}
Conversely, if $(X,\delta)$ is a sober approach space, then for each map $\Phi\colon \mathfrak{L}X\lra[0,\infty]$ that satisfies {\rm (C1)-(C3)}, there is a compact saturated function $\lambda$ of $(X,\delta)$ such that $\Phi(\phi)=\phi\cpt\lambda$.
\end{lem}
\begin{proof}
We leave it to the reader to check that for each compact saturated function $\lambda$ of $(X,\delta)$, the map $\Phi(\phi)=\phi\cpt\lambda$ satisfies (C1)-(C3).

For the converse part, suppose that $\Phi\colon \mathfrak{L}X\lra[0,\infty]$   satisfies {\rm (C1)-(C3)}.

First of all, by (C2), (C3) and Proposition \ref{non-expansive via tensor} one sees that $\Phi\colon (\mathfrak{L}X,\rho_X)\lra([0,\infty],d_L)$ is a non-expansive map. We claim that the function   $$\lambda\colon X\lra[0,\infty],\quad \lambda(x)=\Phi(\delta(-,\{x\}))$$
satisfies the requirement; that is, $\lambda$ is a compact saturated function of $(X,\delta)$ and $\Phi(\phi)=\phi\cpt\lambda$.

Since
$$\lambda(x)+\delta(x,\{y\})=\Phi(\delta(-,\{x\}))+\rho_X(\delta(-,\{x\}), \delta(-,\{y\})) \geq\Phi(\delta(-,\{y\}))=\lambda(y),$$
it follows that $\lambda$ is saturated.

Next we show that for each $\phi\in\mathfrak{L}X$, $\Phi(\phi)=\phi\cpt\lambda$. This equality together with (C1) implies that $\lambda$ is  compact. On one hand,
\begin{align*}
\phi\cpt\lambda & = \inf_{x\in X}(\phi(x)+\Phi(\delta(-,\{x\})))\\
& = \inf_{x\in X}(\rho_X(\delta(-,\{x\}),\phi)+\Phi(\delta(-,\{x\}))) &\text{(Equation (\ref{YonedaforLR}))}\\
& \geq \Phi(\phi).
\end{align*}
On the other hand, for each $r\in[0,\infty)$ with $\Phi(\phi)\leq r$, let
$$\mathcal{A}_r=\{\xi\in\mathfrak{L}X\mid\phi\leq\xi, \Phi(\xi)\leq r\}.$$
By (C1) and Zorn's lemma, $\mathcal{A}_r$ has a maximal element, say $\psi$. By (C2) one sees that $\psi$ is an irreducible lower regular function; by (C3) one sees that $\psi\neq\infty_X$.  Since $(X,\delta)$ is sober, there exist a unique $x$ and a unique $s\in[0,\infty)$  such that $\psi=s+\delta(-,\{x\})$. Since $\psi\in\mathcal{A}_r$, it follows that
\begin{align*}\phi\cpt\lambda&\leq\psi\cpt\lambda\\ &= \inf_{z\in X}(s+\delta(z,\{x\})+\Phi(\delta(-,\{z\})))\\ &= s+\inf_{z\in X}\Big(\rho_X(\delta(-,\{z\}),\delta(-,\{x\}))+\Phi(\delta(-,\{z\}))\Big)\\ &=s+\Phi(\delta(-,\{x\})) & \text{($\Phi$ is non-expansive)}\\ &=\Phi(\psi) & \text{(C3)}\\ &\leq r.\end{align*}
Therefore, $\phi\cpt\lambda\leq\Phi(\phi)$ by arbitrariness of $r$. \end{proof}

\begin{proof}[Proof of Theorem \ref{hm}]
(i) Suppose that $\Lambda\colon \mathfrak{U}X\lra[0,\infty]$ is a  proper open $[0,\infty]$-filter of $(\mathfrak{U}X,\rho_X)$. First of all, we assert that $\Lambda$ satisfies that $$\Lambda(\lambda+r)=\Lambda(\lambda)+r$$ for all $\lambda\in\mathfrak{U}X$ and $r\in[0,\infty]$. If $\lambda=\infty_X$, the equality is obvious; if $\lambda$ is bounded, then by Proposition \ref{real-valued filter} we have that $\Lambda(\lambda+r)\geq r$ and \[\Lambda(\lambda) = \Lambda((\lambda+r)\ominus r)= \Lambda(\lambda+r)-r,\]   hence $\Lambda(\lambda+r)=\Lambda(\lambda)+r$.

We turn to the proof of the conclusion. Uniqueness of $\nu$ follows from Corollary \ref{saturation}; as for  existence, define a map $\Phi\colon \mathfrak{L}X\lra[0,\infty]$ as follows: for each lower regular function $\phi$ of $(X,\delta)$, let
$$\Phi(\phi)=\sup_{x\in X}\phi(x)\ominus \Lambda\Big(\sup_{x\in X}\phi(x)\ominus\phi\Big)$$ if $\phi$ is bounded; and let
$$\Phi(\phi)=\sup_{r\in [0,\infty)}\Phi(\inf\{\phi,r_X\})$$
if $\phi$ is not bounded.

Before proceeding, we notice that, as a  consequence of the equality $\Lambda(\lambda+r)=\Lambda(\lambda)+r$, for each bounded $\phi$ and each bound $b$ of $\phi$,   it holds that   \[\Phi(\phi)=b- \Lambda(b-\phi).\]

We claim that $\Phi\colon \mathfrak{L}X\lra[0,\infty]$ satisfies the conditions  (C1)-(C3) in Lemma \ref{HMC}. Here we check (C1) for example. Let $\{\phi_i\}_{i\in I}$ be a directed family of lower regular functions. If $\sup_{i\in I}\phi_i$ is bounded, pick an upper bound  $b$ of all $\phi_{i}$, then
\begin{align*}
\Phi(\sup_{i\in I}\phi_i) & =b\ominus \Lambda(b\ominus\sup_{i\in I}\phi_i)\\
& = b\ominus \Lambda(\inf_{i\in I}b\ominus\phi_i)\\
& = b\ominus \inf_{i\in I}\Lambda(b\ominus \phi_i)\\
& = \sup_{i\in I} b\ominus \Lambda(b\ominus \phi_i)\\
& = \sup_{i\in I}\Phi(\phi_i).
\end{align*}
If $\sup_{i\in I}\phi_i$ is not bounded, we have
\begin{align*}
\Phi(\sup_{i\in I}\phi_i) & = \sup_{r\in[0,\infty)}\Phi(\inf\{\sup_{i\in I}\phi_i,r_X\})\\
& = \sup_{r\in[0,\infty)}\Phi(\sup_{i\in I}\inf\{\phi_i,r_X\})\\
& = \sup_{r\in[0,\infty)}\sup_{i\in I}\Phi(\inf\{\phi_i,r_X\})\\
& = \sup_{i\in I}\sup_{r\in[0,\infty)}\Phi(\inf\{\phi_i,r_X\})\\
& = \sup_{i\in I}\Phi(\phi_i).
\end{align*}

Since $(X,\delta)$ is sober, by Lemma \ref{HMC}, there is a compact saturated function $\nu$ of $(X,\delta)$ such that $\Phi(\phi)=\phi\cpt\nu$. We  show that  $\nu$ satisfies the requirement. First, it follows from $0_X\cpt \nu=\Phi(0_X)=0$  that $\nu$ is inhabited. Second, we show that $\rho_X(\nu,\lambda)= \Lambda(\lambda)$ for all $\lambda\in \mathfrak{U}X$. The equality is trivial if $\lambda=\infty_X$; if $\lambda$ is   bounded, pick a bound $b$ of $\lambda$, then
\begin{align*}
\rho_X(\nu,\lambda) & =\rho_X(\nu,b\ominus(b\ominus\lambda))\\ &= b\ominus ((b\ominus\lambda)\cpt\nu) &\text{(Lemma \ref{sub vs tensor})}\\
& = b\ominus\Phi(b\ominus\lambda)\\
& =  b\ominus( b-\Lambda( b- (b-\lambda))) &\text{($b$ is a bound of $b-\lambda$)} \\
&= b-( b-\Lambda(\lambda)) &(\Lambda(\lambda)\leq b)\\
& = \Lambda(\lambda).
\end{align*}

(ii) Suppose $\Lambda\colon \mathfrak{U}X\lra[0,\infty]$ is a  decent open $[0,\infty]$-filter of $(\mathfrak{U}X,\rho_X)$. Let
$$s=\sup\{r\in[0,\infty)\mid \Lambda(r_X)=0\}.$$
Since $\Lambda$ is decent,  $s<\infty$. Then for each $r\in[0,\infty)$, $\Lambda(r_X+s)=r$. Hence the map
$$\Lambda'\colon \mathfrak{U}X\lra[0,\infty],\quad \Lambda'(\lambda)=\Lambda(\lambda+s)$$
is a  proper open $[0,\infty]$-filter of $(\mathfrak{U}X,\rho_X)$. By (i), there exists a unique inhabited compact saturated function $\nu$ of $(X,\delta)$ such that  $\Lambda'=\rho_X(\nu,-)$. Since for each $\lambda\in\mathfrak{U}X$ we have
\begin{align*}
\Lambda(\lambda) & = \Lambda((\lambda+s)\ominus s)\\
& = \Lambda(\lambda+s)\ominus s\\
& = \Lambda'(\lambda)\ominus s\\
& = \rho_X(\nu,\lambda)\ominus s\\
& = \rho_X(\nu+s,\lambda),
\end{align*}
then $\mu\coloneqq\nu+s$ is a compact saturated function of $(X,\delta)$ such that $\Lambda=\rho_X(\mu,-)$.

Finally, we show that the bijections in (i) and (ii) are isomorphisms between metric spaces. It suffices to show that for all  saturated functions $\theta$ and $\lambda$,
$$\rho_X(\theta,\lambda)=\rho_{\mathfrak{U}X}(\rho_X(\lambda,-),\rho_X(\theta,-)).$$
For this we only need to check that $\rho_X(\theta,\lambda)\leq\rho_{\mathfrak{U}X}(\rho_X(\lambda,-),\rho_X(\theta,-))$. To this end, we show that for each $r\in(0,\infty)$, if $r<\rho_X(\theta,\lambda)$, then $r\leq\rho_{\mathfrak{U}X}(\rho_X(\lambda,-),\rho_X(\theta,-))$. Assume that $r<\rho_X(\theta,\lambda)$, then there exists an $x\in X$ such that $r\leq\lambda(x)\ominus\theta(x)$. If $\lambda(x)<\infty$, consider the upper regular function $\mu\coloneqq\lambda(x)\ominus\delta(-,\{x\})$, we have
\begin{align*}
\rho_X(\theta,\mu)
& = \lambda(x)\ominus(\delta(-,\{x\})\cpt\theta) & \text{(Lemma \ref{sub vs tensor})}\\
& = \lambda(x)\ominus\theta(x) & \text{(Equation \ref{YonedaforRL})}\\
& \geq r.
\end{align*}
A similar calculation gives that
$\rho_X(\lambda,\mu)=\lambda(x)\ominus\lambda(x)=0$.
Therefore,
\[r\leq\rho_X(\theta,\mu)\ominus\rho_X(\lambda,\mu)\leq\rho_{\mathfrak{U}X}(\rho_X(\lambda,-),\rho_X(\theta,-)).\]
If $\lambda(x)=\infty$, then $\theta(x)<\infty$, consider the upper regular function $(\theta(x)+r)\ominus\delta(-,\{x\})$, we have
\[\rho_X(\theta,(\theta(x)+r)\ominus\delta(-,\{x\}))=r,\]
and
\[\rho_X(\lambda,(\theta(x)+r)\ominus\delta(-,\{x\}))=0.\]
Hence $r\leq\rho_{\mathfrak{U}X}(\rho_X(\lambda,-),\rho_X(\theta,-)).$ \end{proof}

The   Hofmann-Mislove theorem for topological spaces is contained in Theorem \ref{hm}. To see this, let $X$ be a sober topological space, and let $\mathcal{F}$ be a proper open filter of the open-set lattice of $X$.
It is readily verified that
$$\Lambda\colon \mathfrak{U}(\omega(X))\lra[0,\infty],\quad \Lambda(\lambda)=\inf\{\rho_X(0_U,\lambda)\mid U\in\mathcal{F}\}$$
is a proper open $[0,\infty]$-filter of the metric space of upper regular functions of the  approach space $\omega(X)$. Since $\omega(X)$ is sober, there is a unique  inhabited and compact saturated function $\mu$ of $\omega(X)$ such that  $$\Lambda(\lambda)=\rho_X(\mu,\lambda)$$ for all  $\lambda\in\mathfrak{U}(\omega(X))$. We leave it to the reader to check that $\mu$ takes value either $0$ or $\infty$. Let $$Q = \{x\in X\mid \mu(x)=0\}.$$ Then $Q$ is a compact saturated subset of $X$ such that for each open set $U$ of $X$,
\(Q\subseteq U\) if and only if \(  U\in\mathcal{F} \). This gives the Hofmann-Mislove theorem for topological spaces:

\begin{cor}{\rm(Hofmann-Mislove theorem for topological spaces)}
Let $X$ be a sober topological space, the ordered set of open filters of the open-set lattice and the ordered set (ordered by reverse inclusion) of compact saturated subsets of $X$ are isomorphic.
\end{cor}

We end this paper with an example which shows that, unlike in the classical context \cite[page 147]{Gierz2003}, sobriety is not indispensable in Theorem \ref{hm}.

\begin{example} Let $X=([0,\infty],\Gamma(d_L))$ and let $Y$ be the subspace
 $(0,\infty]$ of $X$.

The function
$$\psi\colon (0,\infty]\lra[0,\infty], \quad \psi(x)=0.$$
is  a $[0,\infty]$-ideal of the metric space $((0,\infty],d_L)$, hence an inhabited and irreducible lower regular function of $Y$, but $\psi\neq\Gamma(d_L)(-,\{x\})$ for any $x\in (0,\infty]$, which shows that $Y$ is not sober.

For each function $\theta\colon[0,\infty]\lra[0,\infty]$, denote its restriction on $(0,\infty]$ by $\theta^\dagger$.
Since every coweight of $\dl$ is continuous at $0$, one sees that
$$(-)^\dagger\colon\mathfrak{U}X\lra\mathfrak{U}Y,\quad \lambda\mapsto\lambda^\dagger$$
is an isomorphism of metric spaces. Thus, for each  proper open $[0,\infty]$-filter $\Lambda$  of $\mathfrak{U}Y$,
$$\Lambda'\colon\mathfrak{U}X\lra[0,\infty],\quad \Lambda'(\lambda)=\Lambda(\lambda^\dagger)$$
is a proper open $[0,\infty]$-filter of $\mathfrak{U}X$. Since the metric space $\dl$ is Smyth complete, $X=([0,\infty],\Gamma(d_L))$ is   sober    \cite{LiZ18a}. Then there is an inhabited compact saturated function $\nu$ of $X$ such that $\Lambda'=\rho_X(\nu,-)$. Since $\nu$ is continuous at $0$,  it follows that $\nu^\dagger$ is an inhabited compact saturated function of $Y$ and $\Lambda=\rho_Y(\nu^\dagger,-)$. This shows that the metric space of proper  open $[0,\infty]$-filters of   upper regular functions of $Y$ and the opposite of the metric space of   inhabited  compact saturated functions of $Y$ are  isomorphic. \end{example}

\vskip 8pt

\noindent Junche Yu\\
School of Mathematical Sciences, Beihang University, Beijing, China, cqyjc@icloud.com

\vskip 5pt

\noindent Dexue Zhang\\
School of Mathematics, Sichuan University, Chengdu, China, dxzhang@scu.edu.cn

\end{document}